\documentclass[12pt]{article}

\RequirePackage{natbib}

\usepackage{natbib}

\usepackage{subfig}
\usepackage{setspace,graphicx,epstopdf,amsmath,amsfonts,amssymb,amsthm,versionPO}
\usepackage{marginnote,datetime,enumitem,rotating,fancyvrb}
\usepackage{hyperref,float}

\newdateformat{mydate}{\monthname[\THEMONTH] \THEYEAR} 

\usepackage{caption}

\excludeversion{notes}      
\includeversion{links}          

\iflinks{}{\hypersetup{draft=true}}

\ifnotes{%
\usepackage[margin=1in,paperwidth=10in,right=2.5in]{geometry}%
\usepackage[textwidth=1.4in,shadow,colorinlistoftodos]{todonotes}%
}{%
\usepackage[margin=1in]{geometry}%
}

\makeatletter\let\chapter\@undefined\makeatother 

\usepackage{indentfirst} 

\usepackage{subfig}
\usepackage{versionPO}
\usepackage{caption}
\usepackage{tikz}
\usepackage{tkz-graph}
\DeclareGraphicsExtensions{.pdf, .jpg}
\usetikzlibrary{arrows,%
                petri,%
                topaths}%
\usepackage{tkz-berge}
\usepackage{verbatim}

\newtheorem{theorem}{Theorem}[section]

\newtheorem{corollary}{Corollary}

\newtheorem{definition}{Definition}

\newtheorem{lemma}{Lemma}

\newtheorem{prop}{Proposition}

\usepackage{wasysym}

\makeatletter
\newcommand\varitem[1]{\item[\textbf{A\arabic{enumi}\rlap{$#1$}.}]%
  \edef\@currentlabel{A\arabic{enumi}{$#1$}}}
\makeatother

\usepackage{url}

\begin{document}

\setlist{noitemsep}  

\title{A 2-Dimensional Functional Central Limit Theorem for Non-stationary Dependent Random Fields\footnotetext{$^{\dagger}$Department of Economics, University of Central Florida.  
Email: michael.tseng@ucf.edu}}

\author{Michael C. Tseng
}


\date{
Department of Economics\\
University of Central Florida\\[2ex]
 \mydate   \today
}
\renewcommand{\thefootnote}{\fnsymbol{footnote}}

\singlespacing

\maketitle

\vspace{0.5cm}

\noindent \rule{6.5in}{1pt}

\vspace{-.2in}
{\flushleft \bf Abstract}

\vspace{0.5cm}

\noindent We obtain an elementary invariance principle for multi-dimensional Brownian sheet where the underlying random fields are not necessarily independent or stationary. 
Possible applications include unit-root tests for spatial as well as panel data models.

\medskip
 \vspace{0.5cm}



\thispagestyle{empty}

\clearpage

\onehalfspacing
\setcounter{footnote}{0}
\renewcommand{\thefootnote}{\arabic{footnote}}
\setcounter{page}{1}


\section{Introduction}

We obtain an elementary invariance principle for dependent random fields that does not require stationarity. 
Invariance principles under stationarity and mixing-type of conditions have also been 
considered by \cite{berkes1981strong}, \cite{wang2013new}, and \cite{volny2014invariance}.
Our current setting requires elements of the underlying field $\{ x_{i,j} \}_{i \geq 1, j \geq 1}$ to
have the property that $x_{i,j}$ and $x_{i',j'}$ are uncorrelated whenever $\max\{|i-i'| , |j - j'|\}$ is greater than some finite integer, say $m$.
Such random fields will be referred to as {\sl $m$-dependent}. In dimension one, $m$-dependence generalizes finite-order moving-average time series whose the underlying innovation is 
a martingale difference sequence.
For ease of exposition, we present our result for dimension two. Extending to higher dimensions is straightforward.

The rest of this paper is organized as follows.
Section~\ref{sec: preliminary facts} states preliminary facts regarding tightness and asymptotic Gaussianity for certain random elements on the Skorohod space $D([0,1]^2)$.
Section~\ref{sec: m-dependent fields} defines $m$-dependent random fields and establishes a maximal inequality.
Section~\ref{sec: invariance principle} proves the main result, an invariance principle for $m$-dependent random fields.

\section{Preliminary Facts}
\label{sec: preliminary facts}
\subsection{$D([0,1]^2)$}

We recall relevant properties of the Skorohod metric space $D =  D([0,1]^2)$ (see \cite{bickel1971convergence}).
A {\sl step function} on $[0,1]^2$ is an indicator function of the form $1_{E_1 \times E_2}$ where $E_i$ is either a
left-closed, right-open subset of $[0,1]$ or $\{1\}$.
As a set, $D$ is the uniform closure of the vector space generated by step functions.
Let $\pi$ denote a rectangular
partition of the form $\{ \pi_{ij} = [ \pi_{1 i},  \pi_{1 i+1} ) \times [ \pi_{2 j},  \pi_{2 j+1} ), 1 \leq i \leq n_1, 1 \leq j \leq n_2 \}$
of $[0,1]^2$ where the south and west (resp. north and east) edges of each rectangle are closed (resp. open). 
Let $G_{\delta}$ denote the family of all rectangular
partition $\pi$ such that where $\max \{ | \pi_{1 i} -  \pi_{1 i+1} |,  | \pi_{2 j} -  \pi_{2 j+1} | \} > \delta$ for all $i, j$.
For $x \in D$ and $0 < \delta < 1$, define
$$
w'(x,  \delta ) = \inf_{\pi \in G_{\delta}} \max_{ \bf{s}, \bf{t} \in \pi_{ij} } | x(\bf{s}) - x(\bf{t}) |.  
$$
As in the case of $D[0,1]$, $x$ lies in $D$ if and only if $\lim_{\delta \rightarrow 0} w'(x,  \delta) \rightarrow 0$.
A related quantity, the modulus of continuity,  is defined by
$$
w(x,  \delta ) = \sup_{ \| \bf{s} - \bf{t} \|_{\infty} < \delta } | x(\bf{s}) - x(\bf{t}) |. 
$$ 
It is clear that $x$ lies in $C = C([0,1]^2)$ if and only if $\lim_{\delta \rightarrow 0} w(x,  \delta) \rightarrow 0$.
In general, $w'(x,  \delta ) \leq w(x, 2\delta )$.
If $x \in C$, then $w(x,  \delta ) \leq 2 w'(x, \delta )$.

The Skorohod topology on $D$ is defined as follows.
Let $\Lambda$ denote the class of maps $\lambda : [0,1]^2 \rightarrow [0,1]^2$
such that $\lambda(t_1, t_2) = (\lambda_1(t_1), \lambda_2(t_2))$ 
where $\lambda_1, \lambda_2: [0,1] \rightarrow [0,1]$ are strictly increasing, continuous, 
$\lambda_1(0) = \lambda_2(0) = 0$, $\lambda_1(1) = \lambda_2(1) = 1$.
For $x, y \in D$, the Skorohod metric $d(x,y)$ is defined by
$$
d(x,y) = \inf \{ \epsilon: \, \sup_{ \bf{t} } \| \lambda( \bf{t} ) - \bf{t} \|_{\infty} \leq \epsilon, \sup_t | x (t) - y( \lambda(t) ) | \leq \epsilon \}.
$$
The Skorohod topology coincides with the uniform topology on $C \subset D$.

$D$ is not complete under the Skorohod metric $d$.
The Skorohod topology is also induced by another metric, under which $D$ is complete. 
If one restricts $\lambda$ to maps satisfying
$$
\max_{i = 1, 2} \|\lambda_i\| < \infty, \;   \|\lambda_i\| \equiv \sup_{t_i > s_i} \log \frac{ \lambda_i t_i - \lambda_i s_i   }{t_i - s_i},
$$
the Billingsley's metric $d_0$ is defined by
$$
d_0(x,y) = \inf \{ \epsilon: \,\| \lambda \| \leq \epsilon, \sup_t | x (\bf{t}) - y( \lambda(\bf{t}) ) | \leq \epsilon \}.
$$
From the Taylor expansion estimate
$$
\log 1 - 2 \epsilon \leq -\epsilon \leq \log \frac{\lambda_i t_i }{ t_i} | \leq \epsilon \leq \log 1 + 2 \epsilon,
$$
it follows immediately 
the definition that $d(x,y) \leq 2 d_0(x,y)$.
On the other hand, if $d(x,y) < \delta^2$ for $0 < \delta < \frac14$, then\footnote{This can be proved by applying the same argument as in Lemma 2 on p113 of \cite{billingsley1968convergence} to
each coordinate.}
$$
d_0(x,y) \leq 4 \delta + w'(x, \delta).
$$
Therefore the metric $d_0$ is equivalent to $d$.

The argument for completeness of $D$ under $d_0$ is 
the same as that in Theorem 14.2 on p115 of \cite{billingsley1968convergence}, applied to each
coordinate. The difference between $d_0$ and $d$ is the extra rigidity requirement on $\lambda$, which
implies that certain sequences which are Cauchy under $d$ are not Cauchy under $d_0$.
Next we have a characterization of compactness in $D$ of Arzel\`{a}-Ascoli type.

\begin{prop} 
\label{prop: compactness characterization}
$A \subset D$ is precompact if and only if the following conditions hold:

(i) $\sup_{x \in A} \sup_t  |x(\bf{t})| < \infty$.

(ii) $\lim_{\delta \rightarrow 0} \sup_{x \in A} w'(x, \delta) = 0$.
\end{prop}

This can be proved by applying the same argument as Theorem 14.3 of \cite{billingsley1968convergence}.
The Borel $\sigma$-algebra $\mathcal{D}$ on $D$ is
generated by the coordinate maps $\pi_{ \bf{t}_1, \bf{t}_2, \cdots \bf{t}_k}(x) = ( x(\bf{t}_1), \cdots,  x(\bf{t}_k) )$, $\bf{t}_1, \cdots, \bf{t}_k \in[0,1]^2$. 
Proposition~\ref{prop: compactness characterization} immediately leads to the following characterization of tightness on $D$.

\begin{prop} 
\label{prop: tightness characterization}
A sequence of probability measures $\{ P_n\}$ on $( D, \mathcal{D} ) $ is tight if and only if the following conditions hold:

(i) The family $\{ P_n\}$  pushed forward to the real line by $\|\cdot\|_{\infty}$ is tight, i.e. for all $\eta > 0$, there exists
$a > 0$ such that
$$
P_n (x: \| x \|_{\infty} > a) < \eta, \; \forall n \geq 1.
$$

(ii) For all $\epsilon > 0$ and $\eta > 0$, there exists a $\delta \in (0,1)$ and $n_0$ such that
$$
P_n (x: w'(x, \delta) \geq \epsilon) \leq \eta, \; \forall n \geq n_0.
$$

\end{prop}

See Theorems 8.2 and 15.2 of \cite{billingsley1968convergence}; the same argument as Theorem 8.2 goes through.
Necessity is immediate consequence of Prop.~\ref{prop: compactness characterization}.
Necessity implies that, in condition (ii), $n_0$ can be taken to be equal $1$ without loss of generality, since any finite set
of probability measures is tight.
From this strengthened condition (ii), sufficiency follows.
We are interested in (limit) probability measures whose support lie in $C = C([0,1]^2)$.
The following is the two dimensional analogue of Theorem 15.5 of \cite{billingsley1968convergence}.
It provides sufficient conditions that guarantee tightness as well as any limit measure having support in $C$.

\begin{prop} 
\label{prop: tightness from C conditions}
A sequence of probability measures $\{ P_n\}$ on $( D, \mathcal{D} ) $ is tight if the following conditions hold:

(i)  For all $\eta > 0$, there exists
$a > 0$ such that
$$
P_n (x: x(0,0) > a) < \eta, \; \forall n \geq 1.
$$

(ii) For all $\epsilon > 0$ and $\eta > 0$, there exists a $\delta \in (0,1)$ and $n_0$ such that
$$
P_n (x: w(x, \delta) \geq \epsilon) \leq \eta, \; \forall n \geq n_0.
$$

Moreover, for any $P$ is a weak limit point of $\{P_n\}$, $P(C)=1$.
\end{prop}

\begin{proof}
Since $w'(x, \delta) \leq w(x, 2 \delta)$, condition (ii) of Proposition~\ref{prop: tightness characterization} follows from condition (ii).
Let $\epsilon > 0$ and $\eta > 0$, choose a $\delta \in (0,1)$ such that
$$
P_n (x: w(x, \delta) \leq \epsilon) \geq 1- \frac{\eta}{2}, \; \forall n \geq 1,
$$
and $a > 0$ such that
$$
P_n (x: x(0,0) \leq a) \geq  1- \frac{\eta}{2}, \; \forall n \geq 1.
$$
We have 
$$
P_n (\, (x: w(x, \delta) \leq \epsilon) \cap (x: x(0,0) \leq a) \,) \geq 1 - \eta,
$$
and (partitioning $[0,1]^2$ into $\delta \times \delta$ regular grids),
$$
(x: w(x, \delta) \leq \epsilon) \cap (x: x(0,0) \leq a) \subset (x:  \| x \|_{\infty} \leq a + \epsilon \cdot \frac{1}{\delta^2}).
$$
So condition (i) of Proposition~\ref{prop: tightness characterization} holds.
This proves tightness.

If $w(y, \frac{\delta}{2}) \geq 2 \epsilon$, then $y$ is interior to $w(x, \delta) \geq \epsilon$.
By characterization of weak convergence, a subsequence $P_{n'} \Rightarrow P$ therefore implies
$$
P(y: w(y, \frac{\delta}{2}) \geq 2 \epsilon) \leq \lim \inf_{n'} P_{n'}(x: w(x, \delta) \geq  \epsilon).
$$
Let $\epsilon_k \rightarrow 0$, condition (ii) implies that there exists a sequence $\delta_k \rightarrow 0$ such that
$$
P(y: w(y, \delta_k) \leq  \epsilon_k) \rightarrow 1.
$$
Let $A = \lim \sup_k  \{ y: w(y, \delta_k) \leq  \epsilon_k \}$, then $A \subset C$ and $P(A) = 1$.
This proves the proposition.
\end{proof}

Following Theorem 8.3 of \cite{billingsley1968convergence}, we obtain a sufficient condition
for condition (ii) of Proposition~\ref{prop: tightness from C conditions}
that can be applied to the random elements we will consider.

\begin{prop} 
\label{prop: operational suff cond}
If for all $\epsilon > 0$ and $\eta > 0$, there exists a $\delta \in (0,1)$ and $n_0$ such that
for all $(t_1, t_2) \in [0,1]^2$,
$$
\frac{1}{\delta^2} P_n (x: \sup_{ t_1 \leq s_1 \leq t_1 + \delta,\, t_2 \leq s_2 \leq t_2 + \delta }   | x(s_1, s_2) - x(t_1, t_2) |  \geq \epsilon)  \leq \eta, \; \forall n \geq n_0,
$$
then Condition (ii) of Proposition~\ref{prop: tightness from C conditions} holds.
\end{prop}

A proof can be obtained extending that for Theorem 8.3 of \cite{billingsley1968convergence} to the two dimensional setting.
We now apply the above results to random elements.
Let $\{ \xi_{i,j} \}$ be a field of random variables defined on a common probability space $(\Omega, \mathcal{F}, P)$.
Define the double partial sum by
$$
S_{n_1, n_2} = \sum_{i, j = 1}^{  n_1, n_2} \xi_{i,j}.
$$
Consider the random element in $D$ defined by 
$$
X_n(t_1, t_2) = \frac{1}{n} S_{ [n t_1], [n t_2]}.
$$

\begin{prop} 
\label{prop: operational suff cond for random element}
The sequence $\{ X_n \}$ is tight on $(D, \mathcal{D})$ if
for all $\epsilon > 0$, there exists $\lambda > 0$ and $n_0$ such that
for all $n \geq n_0$ and all $k_1, k_2 \geq 1$,
$$
P(\max_{i \leq n, j \leq n}      | S_{k_1 + i, k_2 + j} - S_{k_1, k_2} |  \geq \lambda n ) \leq \frac{\epsilon}{\lambda^2}.
$$
Moreover, if $P$ is the weak limit of a subsequence of $\{ X_n \}$, $P(C) = 1$.
\end{prop}

Condition (i) of \ref{prop: tightness from C conditions} holds by definition of $X_n$.
Condition (ii) of \ref{prop: tightness from C conditions} applied to $X_n$ says that
$$
\frac{1}{\delta^2} P(\max_{i \leq n \delta, j \leq n \delta} \frac{    | S_{k_1 + i, k_2 + j} - S_{k_1, k_2} | }{n} \geq \epsilon ) \leq \eta
$$
for all $(k_1, k_2)$ and $n$ uniformly large.
Put $m = n \delta$, the expression becomes
$$
P(\max_{i \leq m, j \leq m} | S_{k_1 + i, k_2 + j} - S_{k_1, k_2} |  \geq \frac{m}{\delta} \epsilon ) \leq \delta^2 \eta.
$$
Let $\lambda = \frac{\epsilon}{\delta}$, then the expression becomes
$$
P(\max_{i \leq m, j \leq m} | S_{k_1 + i, k_2 + j} - S_{k_1, k_2} |  \geq m \lambda ) \leq \frac{\epsilon^2}{\lambda^2} \eta.
$$
$\epsilon^2 \eta$ can be collapsed into $\epsilon$ and we arrive at
$$
P(\max_{i \leq m, j \leq m}     \frac{ | S_{k_1 + i, k_2 + j} - S_{k_1, k_2} | }{m} \geq \lambda ) \leq \frac{\epsilon}{\lambda^2},
$$
which is what appears in Proposition~\ref{prop: operational suff cond for random element}.

\begin{corollary} 
\label{cor: running max square UI}
The sequence $\{ X_n \}$ is tight and any limiting measure is supported on $C$ if
the family 
$$
\{ \max_{i \leq n, j \leq n}  \frac{ | S_{k_1 + i, k_2 + j} - S_{k_1, k_2} |^2 }{n^2}, \, n \geq 1, \, k_1, k_2 \geq 1 \}
$$ 
is uniformly integrable.
\end{corollary}

\subsection{Asymptotic Gaussianity}
\label{sec: asymp Gaussianity}
We now give conditions under which a sequence of random elements in $D$ converges to the Brownian sheet in finite dimensional distributions.
The condition below extends those in Section 19 of \cite{billingsley1968convergence}.
For $(s_1, t_1] \times (s_2, t_2] \subset [0,1]^2$, define
$$
X(\Delta_{(s_1, t_1] \times (s_2, t_2]}) = X(t_1, t_2) - X(s_1, t_2) - X(t_1, s_2) + X(s_1, s_2).
$$
Also, for $( {\bf s}, \bf{t} ) \in (0,1)^2$, define 
$$
\mathcal{I}( \bf{s} , \bf{t} ) = \{ ( \bf{s}', \bf{t}') \in (0,1)^2: \, \bf{s}' < \bf{s} \; \mbox{or} \; \bf{t}' < \bf{t} \}. 
$$

The following conditions on infinitesmal moments are sufficient for Gaussianity in the limit.

Condition $1^{\circ}a.$ 
Let $t \in (0,1)$, $( \hat{s}, \hat{t} ] \subset [0,1]$, 
${\bf t}_1, \cdots, {\bf t}_k \subset \mathcal{I}( t, \hat{s} )$.
For all $u_1, \cdots u_k \in \mathbb{R}$,
$$
\limsup_{n \rightarrow 0} \frac{1}{h} | E[  e^{\sum_{j = 1}^k i u_j X_n({\bf t}_j)} X_n( \Delta_{( \hat{s}, \hat{t} ] \times ( t, t+h] }  ) ] | = 0,
$$
$$
\limsup_{n \rightarrow 0} \frac{1}{h} | E[  e^{\sum_{j = 1}^k i u_j X_n({\bf t}_j)} ( X_n^2(\Delta_{( \hat{s}, \hat{t} ] \times ( t, t+h] } ) - h \cdot (\hat{t}- \hat{s}) ) ] | = 0,
$$

Condition $2^{\circ}a.$ 
$$
\lim_{\alpha \rightarrow \infty} \sup_{  {\bf t} \in [0,1]^2  } \limsup_{n \rightarrow \infty} E[ X_n( {\bf t} ) 1_{ \{ X_n( {\bf t} ) \geq \alpha \} }  ] \rightarrow 0.
$$

Condition $3^{\circ}a.$
Let $t \in (0,1)$, $( \hat{s}, \hat{t} ] \subset [0,1]$,
$$
\lim_{\alpha \rightarrow \infty} \limsup_{h \rightarrow 0} \limsup_{n \rightarrow 0} \frac{1}{h} E[ X_n^2(\Delta_{( \hat{s}, \hat{t} ] \times ( t, t+h] } ) 1_{ \{ X_n^2(\Delta_{( \hat{s}, \hat{t} ] \times ( t, t+h] }) \geq \alpha h \} }  ] \rightarrow 0.
$$

The following proposition can be shown by induction, on $k$, and the Cramer-Wold device.

\begin{prop}
Let $W$ denote the Brownian sheet on $D$.
Suppose a sequence of random elements $\{ X_n \}$ satisfies Conditions $1^{\circ}a$, $2^{\circ}a$, and $3^{\circ}a$, then 
$$
P_n \circ \pi_{ \bf{t}_1, \bf{t}_2, \cdots \bf{t}_k}^{-1} \Rightarrow W \circ \pi_{ \bf{t}_1, \bf{t}_2, \cdots \bf{t}_k}^{-1}
$$ 
for all $\bf{t}_1, \cdots, \bf{t}_k \in[0,1]^2$.    
\end{prop}

The lemma below summarizes the fact that, under weak convergence of finite dimensional law and tightness, Brownian sheet is the unique weak limit.

\begin{lemma}
\label{lemma: wk conv to Brownian sheet}
Suppose a sequence of random elements $X_n, n = 1, 2, \cdots \subset D$ is tight and $P(X \in C) = 1$ for any weak limit $X$.
If $\{ X_n \}$ satisfies Conditions $1^{\circ}a$, $2^{\circ}a$, and $3^{\circ}a$, then $\{X_n\}$ converges weakly to the Brownian sheet.   
\end{lemma}

\section{$m$-dependent Fields}
\label{sec: m-dependent fields}

On $\mathbb{N}^2$, consider the component-wise order defined by $(i, j) \leq (i', j')$ if $i \leq i'$ and $j \leq j'$.

\begin{definition}
(\cite{walsh1986martingales}, p336)

Let $(\Omega, \mathcal{F}, P)$ be a probability space, a family $\sigma$-subalgebras  $\{ \mathcal{F}_{i,j} \}_{i \geq 1, j \geq 1}$ is a ($2$-dimensional) {\sl filtration} 
if $\mathcal{F}_{i,j} \subset \mathcal{F}_{i',j'}$ whenever $(i, j) \leq (i', j')$.

A family of random variables $\{ x_{i,j} \}_{i \geq 1, j \geq 1} \subset L^1( \Omega, \mathcal{F}, P)$ is a {\sl martingale with respect to filtration $\{ \mathcal{F}_{i,j} \}_{i \geq 1, j \geq 1}$} 
if the following holds:

(i) $x_{i,j}$ is $\mathcal{F}_{i,j}$-measurable for all $(i,j)$,

(ii) for all $(i, j) \geq (i', j')$, $E[x_{i,j} | \mathcal{F}_{i', j'}] = x_{i', j'}$. 

\end{definition}

We will restrict to $L^2$-martingales. The conditional expectation operator $E[ \, \cdot \, | \mathcal{F}_{i, j}]$ will be denoted by
$E_{i,j}[\, \cdot \,]$.
For a random variable $y \in L^2$, define the $L^2$-increment, with respect to a given filtration $\{ \mathcal{F}_{i,j} \}$,
$$
\Delta y (i,j) = E_{i, j}[y] - E_{i, j-1}[y] - E_{i-1, j}[y] + E_{i-1,j-1}[y].
$$
Given an $L^2$-martingale $\{ x_{i,j}\}$, for each element $x_{i,j}$ define
$$
\hat{x}^{i', j'}_{i,j} = \Delta x_{i,j} (i', j').
$$
It follows from the definition of a martingale that 
\begin{equation}
\label{eqn: recovery formula for 2-d martingale}
x_{i,j} = \sum_{(i', j') \leq (i,j)} \hat{x}^{i', j'}_{i,j}, \; a.s.
\end{equation}

Next we define $m$-dependent random fields.

\begin{definition}
A family of random variables $\{ x_{i,j} \}_{i \geq 1, j \geq 1} \subset L^2( \Omega, \mathcal{F}, P)$ is {\sl $m$-dependent with respect to 
filtration $\{ \mathcal{F}_{i,j} \}_{i \geq 1, j \geq 1}$} if the following holds:

(i) $x_{i,j}$ is $\mathcal{F}_{i+m,j+m}$-measurable for all $(i,j)$,

(ii) For any $k > m$, $E_{i-k, j}[x_{i,j}] =  E_{i, j=k}[x_{i,j}] = E_{i-k, j-k}[x_{i,j}] = 0$.
\end{definition}

It follows from the definition that, if $\{ x_{i,j} \}$ is a $m$-dependent field,
\begin{equation}
\label{eqn: recovery formula for m-dependent field}
x_{i,j} = \sum_{-m \leq k_1 \leq m, -m \leq k_2 \leq m} \hat{x}^{i-k_1, j-k_2}_{i,j}, \; a.s.
\end{equation}

\begin{lemma}
\label{lemma: lemma for m-dependent 2}
Let $\{ x_{i,j} \}_{i \geq 1, j \geq 1}$ be $m$-dependent with respect to 
filtration $\{ \mathcal{F}_{i,j} \}$.
For any $-m \leq k_1,  k_2 \leq m$, define   
$\sigma$-subalgebras $\mathcal{G}_{i,j} = \mathcal{F}_{i-k_1,j-k_2}$. 
Then 
$$
Y_{(i,j)}^{(k_1, k_2)} = \sum_{1 \leq i' \leq i, 1 \leq j' \leq j} \hat{x}^{i'-k_1, j'-k_2}_{i',j'}
$$ 
is a martingale with respect to $\{ \mathcal{G}_{i,j} \}$.
\end{lemma}

\begin{proof}
For any $(i',j')$, $\hat{x}^{i'-k_1, j'-k_2}_{i',j'}$ is $\mathcal{G}_{i',j'}$-measurable by definition.
Since $\{ \mathcal{G}_{i,j} \}$ is a filtration, $Y_{(i,j)}^{(k_1, k_2)}$ is $\mathcal{G}_{i,j}$-measurable.
Let $(i', j') \geq (i ,j)$,
$$
E[ \hat{x}^{i'-k_1, j'-k_2}_{i',j'} | \mathcal{G}_{i,j}] = E[ \hat{x}^{i'-k_1, j'-k_2}_{i',j'} | \mathcal{F}_{i-k_1,j-k_2}] = 0.
$$
This proves the martingale property.
\end{proof}

The following theorem due to Walsh extends the $L^p$-inequality to $2$-dimensional martingales.

\begin{theorem}
\label{thm: Doob L^p}(\cite{walsh1986martingales}, p351)
Let $\{ x_{i,j} \}_{i \geq 1, j \geq 1}$ be a martingale, $p > 1$, and $\frac{1}{p} + \frac{1}{q} = 1$.
Define 
$$
S_{n_1, n_2} = \sum_{i, j = 1}^{  n_1, n_2} x_{i,j}, \;\; S^*_{n_1, n_2} = \sup_{ (i, j) \leq (n_1, n_2) } |S_{i,j}|
$$
Then
$$
E[ ( S^*_{n_1, n_2} )^p ] \leq ( \frac{p}{p-1} )^{2p}  \sup_{ (i, j) \leq (n_1, n_2) } E [ |S_{i,j}|^p].
$$
\end{theorem}

For our purposes, we need to extend Theorem~\ref{thm: Doob L^p} to $m$-dependent fields.

\begin{lemma}
\label{lemma: lemma for m-dependent, 2-dimensional case}
Let $\{ z_{ij} \}_{1 \leq i \leq m, 1 \leq j \leq m}$ be complex numbers and $p > 1$, then
$$
| \sum_{-m \leq i, j \leq m} z_{ij} |^{p} \leq ( 2m+1)^{2(p -1)} \sum_{-m \leq i, j \leq m} | z_{ij} |^{p}.
$$
\end{lemma}

\begin{prop}
\label{prop: maximal inequality for m-dependent field}

Let $\{ x_{i,j} \}_{i \geq 1, j \geq 1}$ be an $m$-dependent field, $p > 1$, $\frac{1}{p} + \frac{1}{q} = 1$, and 
$$
Y_{(i,j)}^{(k_1, k_2)} = \sum_{1 \leq i' \leq i, 1 \leq j' \leq j} \hat{x}^{i'-k_1, j'-k_2}_{i',j'}.
$$
Then
$$
E[ ( S^*_{n_1, n_2} )^p ] \leq q^{2p} \cdot (2m+1)^{ 2(p-1)}  \sum_{-m \leq k_1, k_2 \leq m }   \max_{(i,j) \leq (n_1, n_2)} E [ | Y_{(i, j)}^{(k_1, k_2)} |^p ]. 
$$
\end{prop}

\begin{proof}
By Equation~\ref{eqn: recovery formula for m-dependent field},
$$
x_{i,j} = \sum_{-m \leq k_1 \leq m, -m \leq k_2 \leq m} \hat{x}^{i-k_1, j-k_2}_{i,j}, \; a.s.
$$
So, re-arranging the finite sum gives
\begin{align*}
S_{i,j}   &= \sum_{1 \leq i' \leq i, 1 \leq j' \leq j} x_{i',j'} \\
          &= \sum_{1 \leq i' \leq i, 1 \leq j' \leq j} \left( \sum_{-m \leq k_1 \leq m, -m \leq k_2 \leq m} \hat{x}^{i'-k_1, j'-k_2}_{i',j'}  \right)\\
		  &= \sum_{-m \leq k_1 \leq m, -m \leq k_2 \leq m} \left( \sum_{1 \leq i' \leq i, 1 \leq j' \leq j} \hat{x}^{i'-k_1, j'-k_2}_{i',j'}  \right) \\
          &= \sum_{-m \leq k_1, k_2 \leq m}  Y_{(i,j)}^{(k_1, k_2)}.
\end{align*}
By Lemma~\ref{lemma: lemma for m-dependent, 2-dimensional case},
\begin{align*}
| S_{i, j}|^p &=  | \sum_{-m \leq k_1, k_2 \leq m}  Y_{(i,j)}^{(k_1, k_2)} |^p \\
              &\leq ( 2m+1)^{2(p -1)} \sum_{-m \leq i, j \leq m} | Y_{(i,j)}^{(k_1, k_2)} |^{p}.
\end{align*}
This in turn implies
\begin{align*}
E [\max_{ (i, j) \leq (n_1,n_2) } | S_{i, j}|^p] &\leq  E[  \sum_{-m \leq k_1, k_2 \leq m} \max_{(i, j) \leq (n_1,n_2)}  |Y_{(i,j)}^{(k_1, k_2)} |^p ] \\
                                                &=  \sum_{-m \leq k_1, k_2 \leq m} E[ \max_{(i, j) \leq (n_1,n_2)}  |Y_{(i,j)}^{(k_1, k_2)} |^p ] \\
                                                &\leq q^{2p} ( 2m+1)^{2(p -1)} \sum_{-m \leq i, j \leq m} \max_{(i, j) \leq (n_1,n_2)} E[ | Y_{(i,j)}^{(k_1, k_2)} |^{p}],
\end{align*}
where the second inequality follows from applying Theorem~\ref{thm: Doob L^p} to the martingale 
$$
\{ Y_{(i,j)}^{(k_1, k_2)}, {i \leq n_1, j \leq n_2} \}. 
$$
\end{proof}

\section{Invariance Principle}
\label{sec: invariance principle}

\begin{lemma}
\label{lemma: 1}
Let $\{ \nu_{i,j} \}_{i \geq 1, j \geq 1}$ be an $m$-dependent field such that $\{ \nu_{i,j}^2 \}_{i \geq 1, j \geq 1}$ is uniformly integrable and $E[\frac{ S^2_{n, n} }{n^2}] \rightarrow \sigma^2$.
Then the sequence $\{ X_n \}$ is tight and any limiting measure is supported on $C$.
\end{lemma}

\begin{proof}
By Corollary~\ref{cor: running max square UI}, it suffices to show that the family 
$$
\{ \max_{i \leq n, j \leq n}  \frac{ | S_{k_1 + i, k_2 + j} - S_{k_1, k_2} |^2 }{n^2}, \, n \geq 1, \, k_1, k_2 \geq 1 \}
$$ 
is uniformly integrable.
The argument below goes through for any $(k_1, k_2)$ and for notational simplicity, we will omit the $(k_1, k_2)$ subscript (or set them equal to zero).

For $c > 0$ , consider the $m$-dependent fields $\{ \nu_{i,j}^c \}$ and $\{ Z_{i,j} \}$ defined by
\begin{align*}
\nu_{i,j}^c &= E_{i+m, j+m}[\nu_{i,j} 1_{ \{ |\nu_{i,j}| \leq c \} }] -  E_{i-m, j+m}[\nu_{i,j} 1_{ \{ |\nu_{i,j}| \leq c \} }] \\
            &\; - E_{i+m, j-m}[\nu_{i,j} 1_{ \{ |\nu_{i,j}| \leq c \} }] + E_{i-m, j-m}[\nu_{i,j} 1_{ \{ |\nu_{i,j}| \leq c \} }], \\
Z_{i,j}     &= \nu_{i,j} - \nu_{i,j}^c.
\end{align*}
We define the notation
$$
E^y[\xi]\equiv E[ \xi 1_{ \{ |\xi| \geq y \} }  ], \; \overline{\nu^c}_{ij} = \sum_{i' \leq i, j' \leq j} \nu_{i',j'}^c, \; \overline{Z}_{i,j} = \sum_{i' \leq i, j' \leq j} Z_{i',j'}.
$$
Since $S_{i,j}^2 \leq 2  \overline{\nu^c}_{ij}^2 + 2 \overline{Z}_{i,j}^2$,
$$
E^y[ \max_{i \leq n, j \leq n}  \frac{ | S_{i, j} |^2 }{n^2} ] \leq 4 E^{\frac{y}{2}} [ \max_{i \leq n, j \leq n}  \frac{ |\overline{\nu^c}_{ij}|^2 }{n^2} ] + 4 E[ \max_{i \leq n, j \leq n} \frac{ |\overline{Z}_{ij}|^2 }{n^2}].
$$

By uniform square integrability of $\{ \nu_{i,j} \}$, $E[|Z_{ij}|^2] \leq g(c) \rightarrow 0$ as $c \rightarrow \infty$, for all $i,j$.
Therefore, in the notation of Proposition~\ref{prop: maximal inequality for m-dependent field},
$$
E [ | Y_{(i, j)}^{(0, 0)} |^2 ] \leq n^2 g(c),\; \forall (i,j) \leq (n,n),
$$
which implies
$$
E[ \max_{ (i,j) \leq (n,n) } \frac{ |\overline{Z}_{ij}|^2 }{n^2}] \leq 2^4 \cdot (2m+1)^{2}  \sum_{-m \leq k_1, k_2 \leq m }  g(c).
$$
The right-hand side can be made arbitrarily small by choosing $c$ sufficiently large.

We now consider the term $\max_{ (i,j) \leq (n,n)}  \frac{ |\overline{\nu^c}_{ij}|^2 }{n^2}$.
Applying Proposition~\ref{prop: maximal inequality for m-dependent field} to the case $p = 4$ gives
$$
E[ \max_{ (i,j) \leq (n,n) }  |\overline{\nu^c}_{ij}|^4 ] \leq \frac{4}{3}^8 \cdot (2m+1)^{6}  \sum_{-m \leq k_1, k_2 \leq m }  \max_{ (i,j) \leq (n,n) } E [ | Y_{(i, j), (k_1, k_2)} |^4 ], 
$$
where
$$
Y_{(i,j)}^{(k_1, k_2)} = \sum_{ (i',j') \leq (i,j) } E[ \nu^c_{i',j'} | \mathcal{F}_{ i' + k_1, j' + k_2 }].
$$
Each summand in $Y_{(i,j)}^{(k_1, k_2)}$ is bounded in absolute value by $c$ a.s. 
By a similar counting argument as in inequality 23.7 of \cite{billingsley1968convergence}, one can show that
$$
E [ | Y_{(i, j), (k_1, k_2)} |^4 ] \leq n^4 K_c
$$
where $K_c$ is a constant that only depends on $c$.
This shows $\{ \max_{(i,j) \leq (n,n)} \frac{ |\overline{\nu^c}_{ij}|^2 }{n^2}, n \geq 1 \}$ is bounded in $L^2$, therefore uniformly integrable.
One can then choose $y$ sufficiently large so that 
$$
E^{\frac{y}{2}} [\max_{i \leq n, j \leq n}  \frac{ |\overline{\nu^c}_{ij}|^2 }{n^2}]
$$
is sufficiently small, uniform in $n$. 

In other words, for all $\eta > 0$, there exists $c$  such that $ 4 E[ \max_{i \leq n, j \leq n} \frac{ |\overline{Z}_{ij}|^2 }{n^2}] < \frac{\eta}{2}$.
With this given choice of $c$, there exist $y > 0$ such that $E^{\frac{y}{2}} [ \max_{i \leq n, j \leq n}  \frac{ |\overline{\nu^c}_{ij}|^2 }{n^2} ] < \frac{\eta}{2}$.
This proves the proposition. 
\end{proof}

\begin{theorem}
Suppose, in addition to assumptions of Lemma~\ref{lemma: 1}, 
$$
E[\frac{ ( S_{k_1 + n, k_2 + n} - S_{k_1, k_2})^2 }{n^2} | \mathcal{F}_{k_1 - m_1, k_2 - m_2}] \rightarrow \sigma^2,
$$
as $n \rightarrow \infty$ for all $(k_1, k_2)$,
then $\{ X_n \}$ converges weakly to the Brownian sheet.

\end{theorem}

\begin{proof}
According to Lemma~\ref{lemma: wk conv to Brownian sheet}, it suffices to verify Conditions $1^{\circ}a$, $2^{\circ}a$, and $3^{\circ}a$.

Uniform integrability of the family, 
$$
\{ \max_{i \leq n, j \leq n}  \frac{ | S_{k_1 + i, k_2 + j} - S_{k_1, k_2} |^2 }{n^2}, \, n \geq 1, \, k_1, k_2 \geq 1 \}
$$ 
follows from Lemma~\ref{lemma: 1}.
This in turn implies that 
$$
\{ \frac{  X_n^2 (\Delta( ( \hat{s}, \hat{t} ] \times (t, t+h] )) }{ h (\hat{t} - \hat{s})}, h \geq \frac{1}{n}, t \in (0,1-h),\, ( \hat{s}, \hat{t} ] \subset [0,1], \, n = 1, 2, \cdots \}
$$ 
is uniformly integrable.
Condition $2^{\circ}a$ now follows by taking $t = \hat{s} = 0$. 
Condition $3^{\circ}a$ is also immediate (since $0 < \hat{t} - \hat{s} < 1$.)

To verify Condition $1^{\circ}a$, let $t \in (0,1)$, $( \hat{s}, \hat{t} ] \subset [0,1]$, 
${\bf t}_1, \cdots, {\bf t}_k \subset \mathcal{I}( t, \hat{s} )$, and $u_1, \cdots u_k \in \mathbb{R}$.
Define 
$$
U_n = E[\sum_{l = 1}^k i u_l X_n({\bf t}_l)| \mathcal{F}^{wp}_{[n (t,s)]}  ]
$$
where the ``weak past" $\sigma$-algebra $\mathcal{F}^{wp}_{k_1, k_2}$ is defined to be
$$
\mathcal{F}^{wp}_{k_1, k_2} = \sigma(  \mathcal{F}_{i,j}, \, i \leq k_1, \; \mbox{or} \; j \leq k_2).
$$
Then
\begin{align*}
\|  U_n -   \sum_{l = 1}^k i u_l X_n({\bf t}_l) \|_2 & \leq \max \{|u_1|, \cdots |u_k| \} \frac{1}{n \sigma} \sum_{l = 1}^k \sum_{ (1,1) \leq (i,j) \leq [n {\bf t}_l]  }  \| E[ \nu_{i,j}| \mathcal{F}^{wp}_{[n (t,s)]} ] - \nu_{i,j} \|_2 \\
                                                     & \leq K \cdot \frac{1}{n} ,  
\end{align*}
for some constant $K$ that depends on $m$ and $\sup \| \nu_{i,j} \|^2$, since the sum 
$$
\sum_{l = 1}^k \sum_{ (1,1) \leq (i,j) \leq [n {\bf t}_l]  }  \| E[ \nu_{i,j}| \mathcal{F}^{wp}_{[n (t,s)]} ] - \nu_{i,j} \|_2
$$ 
is finite by $m$-dependence and $L^2$-boundedness.

Together with uniform integrability of $ X_n(\Delta( ( \hat{s}, \hat{t} ] \times (t, t+h] ))$, this in turn implies
$$
E[  ( e^{\sum_{j = 1}^k i u_j X_n({\bf t}_j)} - e^{i U_n}) X(\Delta_{t, t+h}) ] \rightarrow 0.
$$
But
$$
| E[ e^{i U_n} X_n(\Delta_{t, t+h}) ] | \leq \| E[ X(\Delta_{t, t+h}) |  \mathcal{F}^{wp}_{[n (t,s)]}] \|_2 = O(  \frac{1}{n}  ) \rightarrow 0
$$
as $n \rightarrow \infty$.
Similarly, uniform integrability of $ X_n^2(\Delta_{t, t+h})$ implies
$$
E[ | ( e^{\sum_{j = 1}^k i u_j X_n({\bf t}_j)} -  e^{i U_n} ) ( X_n^2(\Delta_{t, t+h}) - h(\hat{t}- \hat{s}) ) | ]  \rightarrow 0.
$$
Since
$$
| E[ e^{i U_n}  ( X^2(\Delta_{t, t+h}) - h(\hat{t}- \hat{s}) )  ] | \leq  | ( X_n^2(\Delta_{t, t+h}) - h(\hat{t}- \hat{s}) )  ] | \rightarrow 0,
$$
by $(2)$, both conditions in $1^{\circ}a$ are verified. This proves the theorem.

\end{proof}

\section{Conclusion}

We proved an elementary invariance principle for the Brownian sheet where strong or wide-sense stationarity is not required. It is of interest in applications where
stationarity may be too strong an assumption. An immediately application is unit root testing for spatial models, the detailed discussion of which will be given in a separate paper.

\bibliographystyle{chicago}
\bibliography{Reference}

\end{document}